\documentclass[a4paper,reqno]{amsart}
 
\usepackage{amsthm}
\usepackage{amssymb}
\usepackage{amsmath}
\usepackage{tikz}
\usepackage{fullpage}
\usepackage[english]{babel}
\usepackage{bm}
\usepackage{txfonts}
\usepackage{enumitem}
\usepackage{dsfont}
    
\usepackage{microtype}
\usepackage[colorlinks = true, linkbordercolor = {white}, urlcolor={purple}, linkcolor={purple}, citecolor = {purple}]{hyperref}
\usepackage{cleveref}
\usepackage[foot]{amsaddr}

\theoremstyle{definition}

\numberwithin{equation}{section}

\author[A.\,Bura]{A.\,Bura}
\author[C.\,Good]{C.\,Good}
\author[T.\,Samuel]{T.\,Samuel}

\address[A.\,Bura and C.\,Good]{School of Mathematics, University of Birmingham, UK}
\address[T.\,Samuel]{Department of Mathematics and Statistics, University of Exeter, UK}

\thanks{\emph{Funding acknowledgements}. The \emph{Birmingham-Leiden Collaboration Fund}, and EPSRC grants EP/S02297X/1 and EP/Y023358/1.}

\renewcommand{\emph}{\textsl}

\renewcommand{\textit}{\textsl}

\newtheorem{theorem}{Theorem}[section]
\newtheorem{mainthm}{Theorem}

\theoremstyle{definition}

\theoremstyle{remark}

\numberwithin{equation}{section}



\title{The shadowing property for piecewise monotone interval maps}

\begin{document}

    \begin{abstract}
        The property of shadowing has been shown to be fundamental in both the theory of symbolic dynamics as well as continuous dynamical systems.  A quintessential class of discontinuous dynamical systems are those driven by transitive piecewise monotone interval maps and in particular \mbox{$\beta$-transformations}, namely transformations of the form $T_{\beta, \alpha} : x \mapsto \beta x + \alpha \; (\bmod \, 1)$ acting on $[0,1]$. We provide a short elegant proof showing that this class of dynamical systems does not possess the property of shadowing, complementing and extending the work of Chen and Portela.
    \end{abstract}

\maketitle

\section{Introduction}

Let $f : X \to X$ be a continuous function on a compact metric space $(X, \mathrm{d})$. Letting $\delta$ denote a positive real number, a $\delta$\emph{-pseudo-orbit} of the dynamical system $(X, f)$ is a sequence $(x_i)_{i\in \mathbb{N}_0}$ such that $\mathrm{d}(f(x_{i}),x_{i+1}) <\delta$. Given $\epsilon >0$, a sequence $(y_i)_{i \in \mathbb{N}_{0}}$ in $X$ is said to \emph{$\epsilon$-shadow} the sequence $(x_i)_{i\in \mathbb{N}_{0}}$ if $\mathrm{d}(y_i,x_i)<\epsilon$ for all $i \in \mathbb{N}_{0}$. We say that $f$ has \emph{shadowing} if for every $\epsilon >0 $, there exists a $\delta > 0$, such that every $\delta$-pseudo-orbit is $\epsilon$-shadowed by a true orbit, namely where we take $y_{i} = f^{i}(x)$ for all $i \in \mathbb{N}_{0}$ and some fixed $x \in X$. Additionally, we say that $f$ has finite shadowing if for all $\epsilon > 0$, there exists $\delta > 0$ such that, if $(x_i)_{i \in \{0,\dots,n \}}$ is a finite $\delta$-pseudo-orbit, then $\mathrm{d}(f^{i}(z),x_{i}) < \epsilon$ for some $z$ and for all $i \in \{0, \dots, n\}$.

Shadowing was introduced by Anosov to study hyperbolic dynamical systems, and has gone on to become widely studied in various other branches of mathematics. It features prominently when studying fractals that arise in the field of complex dynamics \cite{barwell-raines-meddaugh,barwell-raines}, and in the study of Axiom A diffeomorphisms \cite{ Bowen, bowen-markov-partitions}. It has also  developed into an important concept in numerical analysis as rounding errors ensure that computed orbits will necessarily be a pseudo-orbit \cite{Corless, Pearson}.

Chen's \cite{84b458cf-12bb-3c39-8e22-eea22417dd6c} work on the shadowing property for continuous uniformly piecewise monotone interval maps, naturally leads us to ask what happens when continuity is no longer a requirement. We show that shadowing fails in the case for transitive piecewise continuous monotone interval maps.  Moreover, we show that all \mbox{$\beta$-transformations}, a class of interval maps that ubiquitous in the study of dynamical systems and which are not necessarily transitive, does not possess the property of shadowing.

More precisely, we say that $f : [0,1] \to [0,1]$ is a \emph{piecewise continuous monotone interval map} if there exist non-empty disjoint intervals $I_0 = (0,z_{0}), I_{1}=(z_{0},z_{1}), \dots, I_{n-1}=(z_{n-2},z_{n-1}), I_n = (z_{n-1},1)$, for some $n \in \mathbb{N}$, such that $f$ is strictly monotone and continuous on each of these intervals, and $\lim_{x \nearrow z_{m}} f(x) \neq \lim_{x \searrow z_{m}} f(x)$ for all $m \in \{ 0, 1, \dots, n-1 \}$. The collection of intervals $I_{0}, \dots, I_{n-1}$ is referred to as a \textsl{partition} of $f$. We say that an interval map $f : [0,1] \to [0,1]$ is \emph{transitive}, if for any open interval $U$ of $[0,1]$, there exists an $m \in \mathbb{N}_{0}$ with $\bigcup_{n = 0}^{m} f^{n}(U) \supseteq (0,1)$.

    \begin{mainthm}\label{thm:Shadowing_Main}
        If $f$ is a transitive piecewise continuous monotone interval map, then $f$ does not have finite shadowing, and hence, does not have shadowing.
    \end{mainthm}

\Cref{thm:Shadowing_Main} extends and complements the results of 
\cite{Coven,84b458cf-12bb-3c39-8e22-eea22417dd6c,Raquel_Ribeiro_Barroso_Portela_Thesis}.  In particular, Coven, Kan and Yorke \cite{Coven} gave a classification of which tent maps, namely maps of the form
    \begin{align*}
        f_k(x)= 
            \begin{cases}
                kx      &\text{if} \; 0 \leq x \leq1,\\
                k(2-x)\quad  &\text{if} \; 1 \leq x \leq 2,
            \end{cases}
    \end{align*}
have shadowing. Chen \cite{84b458cf-12bb-3c39-8e22-eea22417dd6c} extended these results to continuous uniformly piecewise linear interval maps, and Portela \cite{Raquel_Ribeiro_Barroso_Portela_Thesis} then expanded upon this work studying shadowing for discontinuous piecewise linear interval maps. In these works, the authors presented various conditions, such as the \emph{linking property} and \emph{Condition C} to determine whether these maps have shadowing or not. These conditions can be difficult to verify, however, our proof of \Cref{thm:Shadowing_Main} provides a simpler alternative method, proving this result, under the hypothesis of transitivity and without the assumption of linearity.

Moreover, we show how this result can be applied to a widely studied class of expanding interval maps. Namely \mbox{$\beta$-transformations} -- transformations of the form $T_{\beta, \alpha} : x \mapsto \beta x + \alpha \; (\bmod \, 1)$ acting on $[0,1]$, where $(\beta, \alpha)$ belongs to the parameter space $\Delta = \{ (b,a) \in \mathbb{R}^{2} : \beta \in (1,2] \; \text{and} \; \alpha \in [0, 2-\beta] \}$. This class of transformations have motivated a wealth of results, providing practical solutions to a variety of problems. They arise as Poincar\'e maps of the geometric model of Lorenz differential equations~\cite{MR681294}, Daubechies \textsl{et al.} \cite{1011470} proposed a new approach to analog-to-digital conversion using \mbox{$\beta$-transformations}, and Jitsumatsu and Matsumura \cite{Jitsumatsu2016AT} developed a random number generator using \mbox{$\beta$-transformations}. Through their study, many new phenomena have appeared, revealing rich combinatorial and topological structures, and unexpected connections to probability theory, ergodic theory and aperiodic order, see for instance \cite{Komornik:2011,ArneThesis,bezuglyi_kolyada_2003}.

This class of interval maps also has an intimate link to metric number theory in that they give rise to \mbox{non-integer} based expansions of real numbers. Given $\beta \in (1, 2)$ and $x \in [0, 1/(\beta-1)]$, an infinite word $\omega = \omega_{1}\omega_{2}\cdots \in \{0,1\}^{\mathbb{N}}$ with letters in the alphabet $\{0, 1\}$ is called a \textsl{$\beta$-expansion} of $x$ if 
    \begin{align*}
        x = \sum_{n \in \mathbb{N}} \omega_{n} \, \beta^{-n}. 
    \end{align*}
Through iterating the map $T_{\beta, \alpha}$ one obtains a subset $\Omega_{\beta, \alpha}$ of $\{ 0, 1\}^{\mathbb{N}}$ known as the \emph{$(\beta,\alpha)$-shift}, where each $\omega \in \Omega_{\beta, \alpha}$ is a \mbox{$\beta$-expansion}, and corresponds to a unique point in $[\alpha/(\beta-1), 1+\alpha/(\beta-1)]$, see for instance \cite{MR3369224} for further details.  In contrast to integer based expansions (such as binary or based $10$ expansions), for a given $x \in [0,1/(\beta-1)]$, if $\beta \leq (-1+\sqrt{5})/2$, then $x$ has uncountably many non-integer based expansions \cite{Erdos1990}, and if $\beta \in ((-1+\sqrt{5})/2, 2)$, then Lebesgue almost all $x \in [0,1/(\beta-1)]$ have uncountably many non-integer based expansions \cite{Sid_2003}.  It is precisely this property which allows for the aforementioned applications in analogue to digital conversion and random number generation.

Since \mbox{$\beta$-transformations} may not be transitive, \Cref{thm:Shadowing_Main} cannot necessarily be applied in this setting. However, we may find first return maps which are transitive, see \cite{G1990,PalmerThesis}. By combining this result with \Cref{thm:Shadowing_Main}, we show the following.

    \begin{mainthm}\label{thm:beta_shadowing}
        \mbox{$\beta$-transformations} do not have finite shadowing, and hence, do not have shadowing.
    \end{mainthm}

In \cite{MR3937681}, it was shown that the set of \mbox{$\beta$-transformations} in $\Delta$ which are subshifts of finite type is a countable dense set. This in tandem with the \Cref{thm:beta_shadowing}, creates a beautiful contrast to the setting of continuous interval maps \cite{bowen-markov-partitions,Coven,GM2020}.

\subsection*{Outline} \Cref{sec:Proof_Thm_A,sec:Proof_Thm_B} are dedicated to the proofs of \Cref{thm:Shadowing_Main} and \Cref{thm:beta_shadowing} respectively, and in \Cref{sec:conclusion} we present some concluding remarks and open questions.

    \section{Transitive piecewise continuous monotone interval maps do not have shadowing -- A proof of \texorpdfstring{\Cref{thm:Shadowing_Main}}{Theorem A}}\label{sec:Proof_Thm_A}

    \begin{proof}[Proof of \Cref{thm:Shadowing_Main}]
    Let $f : [0,1] \to [0,1]$ denote a transitive piecewise continuous monotone interval map, with partition $I_0 = (0,z_{0}), I_{1}=(z_{0},z_{1}), \dots, I_{n-1}=(z_{n-2},z_{n-1}), I_n = (z_{n-1},1)$, for some $n \in \mathbb{N}$. Denote by $\lvert I_{j} \rvert$ the length of $I_{j}$ and set
        \begin{align*}
            f_{+}(z_{j}) = \lim_{x \searrow z_{j}} f(x)
            \quad \text{and} \quad
            f_{+}(z_{j}) = \lim_{x \nearrow z_{j}} f(x).
        \end{align*}
We assume without loss of generality that there exists $k \in \{ 0, \dots, n-1 \}$ with $f_{+}(z_{k}) = f(z_{k})$, as the situation when $f_{-}(z_{j}) = f(z_{j})$ for all \mbox{$j \in \{ 0, \dots, n-1 \}$} follows with analogous arguments. With this in mind, we divide the proof into two cases, when $f(z_{k}) = f_{+}(z_{k}) \neq 0$ and when $f(z_{k}) = f_{+}(z_{k}) = 0$. Before proving either of these cases, we introduce some notation. For $j \in \mathbb{N}$ and $z \in [0,1]$, let $j_{z} = 0$ if one of the following scenarios hold, and $1$ otherwise.
        \begin{itemize}
            \item If $f^{j}(z) \in \operatorname{int}(I_{i})$ for some $i \in \{ 0, \dots, n \}$ and $f^{j}$ restricted to $I_{i}$ is monotonically increasing.
            \item If $f^{j}(y) = z_{i}$ for some $i \in \{ 0, \dots, n-1 \}$, $f(z_{i}) = f_{+}(z_{i})$ and $f^{j}$ restricted to $I_{i+1}$ is monotonically increasing.
            \item If $f^{j}(y) = z_{i}$ for some $i \in \{ 0, \dots, n-1 \}$, $f(z_{i}) = f_{-}(z_{i})$ and $f^{j}$ restricted to $I_{i}$ is monotonically increasing.
            \item If $f^{j}(y) = 0$ or $f^{j}(y) = 1$ and $f^{j}$ restricted to $I_{0}$ or $I_{n}$, respectively, is monotonically increasing.
        \end{itemize}       
    Namely, $j_{z}$ records if the branch of $f^{j}$ to which $f^{j}(z)$ belongs, is monotonically increasing or decreasing.

        \newpage

    {\bfseries Case\,1.} ($f(z_{k}) = f_{+}(z_{k}) \neq 0$): Let $\eta, \epsilon \in \mathbb{R}_{>0}$ be such that $\eta < \min\{f_{-}(z_{k}), \epsilon \}$, $\epsilon + \eta < \min \{ \lvert I_{j} \rvert : j \in \{ 0, \dots, n \} \}$, $\lvert f(x) - f_{+}(z_{j}) \rvert > \epsilon$ for all $x \in (z_{j}, z_{j}+\epsilon)$, and $\lvert f(z_{j}) - f_{-}(z_{j}) \rvert > \epsilon$ for all $x \in (z_{j}-\epsilon, z_{j})$. By way of contradiction, suppose there exists a $\delta \in (0, \eta)$ such that any finite $\delta$-pseudo orbit can be \mbox{$\epsilon$-shadowed} by a true (finite) orbit.  Since $f$ is transitive, there exist $y \in (f(z_{k})-\delta,f(z_{k}))$ and $m \in \mathbb{N}$ with $f^{m}(y) = z_{k}$.  Consider the finite $\delta$-pseudo orbit $(x_{i})_{i\in \{0, \dots, m+2\}}$, where we set
        \begin{align*}
            x_{0}=z_{k}, \;\,
            x_{1}=y, \;\,
            x_{2}=f(y), \; \dots \;
            x_{m}=f^{m-1}(y), \;\,
            x_{m+1}=f^{m}(y)-(-1)^{m_{y}}\delta=z_{k}-(-1)^{m_{y}}\delta, \;\,
            x_{m+2}=f(z_{k}-(-1)^{m_{y}}\delta).
        \end{align*}
    By our hypothesis, there exists $x^{*} \in [0,1]$ which $\epsilon$ shadows $(x_{i})_{i \in \{0, \dots, m+2\}}$, namely with $\lvert f^{i}(x^{*}) - x_{i} \rvert < \epsilon$ for all $i \in \{0, \dots, m+2\}$.  By construction, $x^{*} \geq x_{0} = z_{k}$, and for any $j \in \{ 0, \dots, m\}$, one of the following scenarios hold:
        \begin{itemize}
            \item $x_{j}$ and $f^{j}(x^{*})$ both belong to the interior of $I_{l}$, for some $l \in \{0,\dots,n-1\}$;
            \item $x_{j} = z_{l}$ and $f(z_{l}) = f_{+}(z_{l})$, for some $l \in \{0,\dots,n-1\}$, in which case $f^{j}(x^{*}) \in \operatorname{int}(I_{l+1})$;
            \item $x_{j} = z_{l}$ and $f(z_{l}) = f_{-}(z_{l})$, for some $l \in \{0,\dots,n-1\}$, in which case $f^{j}(x^{*}) \in \operatorname{int}(I_{l})$;
            \item $f^{j}(x^{*}) = z_{l}$ and $f(z_{l}) = f_{+}(z_{l})$, for some $l \in \{0,\dots,n-1\}$, in which case $x_{j} \in \operatorname{int}(I_{l+1})$;
            \item $f^{j}(x^{*}) = z_{l}$ and $f(z_{l}) = f_{-}(z_{l})$, for some $l \in \{0,\dots,n-1\}$, in which case $x_{j} \in \operatorname{int}(I_{l})$.
        \end{itemize}
    Thus, we have $(-1)^{j_{y}}f^{j}(x^{*}) < (-1)^{j_{y}} x_{j}$ for all $j \in \{1, \dots, m+1 \}$, and hence 
        \begin{align*}
        (-1)^{m_{y}}x_{m+1} < (-1)^{m_{y}}z_{k} \leq (-1)^{m_{y}}f^{m+1}(x^{*}) <  (-1)^{m_{y}}(z_{k} + (-1)^{m_{y}}\epsilon).
        \end{align*}
    However, this implies $\lvert f^{m+2}(x^{*}) - x_{m+2} \rvert > \epsilon$, yielding a contradiction and thus proving the result when $f(z_{k}) \neq 0$.

    {\bfseries Case\,2.} ($f(z_{k}) = f_{+}(z_{k}) = 0$): Let $\eta, \epsilon \in \mathbb{R}_{>0}$ be as in the case when $f(z_{k}) \neq 0$, replacing the condition that $\eta < \min\{f_{-}(z_{k}), \epsilon \}$ by the condition $\eta < \epsilon$. By way of contradiction, suppose there exists a $\delta \in (0, \eta)$ such that any finite $\delta$-pseudo orbit can be \mbox{$\epsilon$-shadowed} by a true (finite) orbit.  By transitivity, there exist $y \in (0, \delta)$ and $m \in \mathbb{N}$ with $f^{m}(y) = z_{k}$.  We may also find $w_{0} \in (0,y)$ and $l_{0} \in \mathbb{N}$ with $f^{l_{0}}(w_{0}) = z_{k}$. Similarly, we may find $w_{1} \in (y,\delta)$ and $l_{1} \in \mathbb{N}$ with $f^{l_{1}}(w_{1}) = z_{k}$. Setting $b = m_{y} \; (\bmod \, 2)$, we let $w=w_{b}$ and $l=l_{b}$ and consider the $\delta$-pseudo orbit $(x_{i})_{i \in \{0, \dots,m+l_{b}+3\}}$, where we set
        \begin{alignat*}{11}
            \;\; x_{0} &= z_{k}, \;\;
            \;\; &x_{1} &= y, \;\;
            \;\; &x_{2} &= f(y),
            \;\; &\dots\dots&
            \dots &x_{m} &= f^{m-1}(y), \;\;
            \;\; &x_{m+1} &= z_{k}, \\\
            \;\; & 
            \;\; &x_{m+2} &= w, \;\;
            \;\; &x_{m+3} &= f(w), \;
            \;\; &\dots\dots&
            &x_{m+l+1} &= f^{l-1}(w), \;\;
            \;\; &x_{m+l+1} &= z_{k} - (-1)^{m_{y}+l_{y}}\delta.
        \end{alignat*}
    Note, $l_{y}$ is recording if the branch of $f^{l}$ to which $f^{l}(w)$ belongs, is monotonically increasing, or monotonically decreasing. Following analogous arguments to the case $f(z_{k}) \neq 0$, one obtains a contradiction to our hypothesis that $f$ has finite shadowing.
    \end{proof}

    \section{Applications to \texorpdfstring{$\beta$}{beta}-transformations -- A proof of \texorpdfstring{\Cref{thm:beta_shadowing}}{Theorem B}}\label{sec:Proof_Thm_B}

Here, we present the proof of \Cref{thm:beta_shadowing}.  For this we require the following. Let $(X, \mathrm{d}_{1})$ and $(Y, \mathrm{d}_{2})$ denote two metric spaces, and let $f: X \to X$ and $g : Y \to Y$.  We say that the dynamical systems $(X, f)$ and $(Y, g)$ are \textsl{conjugate} if there exists a homeomorphism $h : X \to Y$ such that $h \circ f = g \circ h$.

    \begin{theorem}[\cite{G1990,PalmerThesis}]\label{thm:beta_transitive}
        Given $(\beta, \alpha)\in \Delta$ with $T_{\beta,\alpha}$ non-transitive,
        there exists an interval $J \subset [0,1]$, an integer $n>2$ and an $\widehat{\alpha} \in [0, 2-\beta^n]$, such that $T_{\widehat{\alpha},\beta^n}$ is transitive, $T_{\beta,\alpha}^{n}(J)=J$, and $T_{\alpha,\beta}^{n}$ restricted to $J$ is topologically conjugate to $T_{\widehat{\alpha},\beta^n}$ where the conjugation map $h$ is an affine transformation.  Moreover, $T_{\beta,\alpha}^{i}(x) \not\in J$, for all $x \in J$ and all $i \in \{ 1, \dots, n-1 \}$, namely $T_{\alpha,\beta}^{n}$ restricted to $J$ is the first return map of $T_{\alpha,\beta}$ on $J$.
    \end{theorem}

    \begin{proof}[Proof of \Cref{thm:beta_shadowing}]
        If $T_{\beta,\alpha}$ is transitive, then \Cref{thm:Shadowing_Main} gives us our desired result. Therefore, assume that $T_{\beta,\alpha}$ is not transitive. Let $J$, $n$ and $\widehat{\alpha}$ be as in \Cref{thm:beta_transitive}, set $T : J \to J$ to be the restriction of $T^{n}_{\beta,\alpha}$ to $J$, and let $\epsilon \in (0, \lvert J \rvert/2)$.  By way of contradiction suppose that there exists $\delta > 0$, such that any finite $\delta$-pseudo orbit with respect to $T_{\beta,\alpha}$ can be $\epsilon$-shadowed by a true (finite) orbit of $T_{\beta,\alpha}$. 

        Since $T_{\widehat{\alpha},\beta^n}$ is piecewise continuous monotone interval map, \Cref{thm:beta_shadowing} in tandem with the fact that $T$ is conjugate to $T_{\widehat{\alpha},\beta^n}$ by an affine transformation, yields the existence of a finite $\delta$-pseudo-orbit $( w_i )_{i \in \{0, \dots, m\}}$ under $T$, which cannot be $\epsilon$-shadowed by a true (finite) orbit of $T$. 

        Choose $a^* \in J$ with $(a^*-\epsilon, a^*+\epsilon) \subset J$. By transitivity, there exists $y^* \in (T(w_m)-\delta,T(w_m)+\delta)\cap J$ and $k \in \mathbb{N}$ with $T^{k}(y^*)=a^*$ .  We may also find $z^* \in (T(a^*)-\delta, T(a^*)+\delta)\cap J$ and $t \in \mathbb{N}$ with $T^{t}(z^*)=w_0$. 

            \newpage

        With this at hand, consider the finite $\delta$-pseudo-orbit $(x_{i})_{i \in \{ 0, \dots, 2m+k+t+2\}}$ of $T$ given by 
            \begin{alignat*}{15}
                x_{0}&=w_0, \;\;
                &\dots\dots& \;\;
                &x_{m}&=w_m, \;\;\;\;
                &x_{m+1}&=y^*, \;\;\;\;
                &x_{m+2}&=T(y^*), \;\;\;
                &\dots\dots&
                \dots &\!\!x_{m+k}&=T^{k-1}(y^*), \;\;\;\;
                x_{m+k+1}=a^*,\\
                &&&&&
                &x_{m+k+2}&=z^*, \;\;\;\;
                &x_{m+k+3}&=T(z^*), \;\;\;
                &\dots\dots&
                &\!\!x_{m+k+t+1}&=T^{t-1}(z^*),\\
                &&&&&
                &x_{m+k+t+2}&=w_0, \;\;\;\;
                &x_{m+k+t+3}&=w_1, \;\;
                \dots&\dots\dots&
                &\!\!x_{2m+k+t+2}&=w_m.
            \end{alignat*}
        By our hypothesis, there exists $c^* \in [0,1]$ such that, under $T_{\beta,\alpha}$, the finite $\delta$-pseudo-orbit $(v_{i})_{i \in \{0,\dots n(2m+k+t+2)\}}$, where
            \begin{align*}
                v_{i} =
                    \begin{cases}
                        x_{j} & \text{if} \; n = j n \; \text{for some} \; j \in \{0, \dots, 2m+k+t+2 \}\\
                        T^l(x_{(j}) \quad & \text{if} \; i = j n + l \; \text{for some} \; j \in \{0, \dots, 2m+k+t+2 \} \; \text{and} \; l \in \{1, \dots, n-1\},
                    \end{cases}
            \end{align*}
        is $\epsilon$-shadowed by $(T_{\beta, \alpha}^{i}(c^{*}))_{i \in \{0,\dots n(2m+k+t+2)\}}$. Hence, under $T$, the $\delta$-pseudo-orbit $(x_{i})_{i \in \{ 0, \dots, 2m+k+t+2\}}$ is \mbox{$\epsilon$-shadowed} by $(T^{i}(c^{*}))_{i \in \{0,\dots 2m+k+t+2 \}}$.  By assumption $( w_i )_{i \in \{0, \dots, m\}}$ cannot be $\epsilon$-shadowed by a true orbit of $T$, and so $c^* \in [0,1]\setminus J$. By \Cref{thm:beta_transitive}, and since $\lvert T^{m+k+1}_{\beta,\alpha}(c^*) - a^*\rvert < \epsilon$, we have that $T_{\beta,\alpha}^{n(m+k+1+j)}(c^*)$, and hence $T^{m+k+1+j}(c^{*})$, belongs to the interval $J$, for all $j \in \mathbb{N} \cup \{0\}$. However, this would imply that $(x_{i})_{i\in \{m+k+t+2, \dots, m+k+t+2 \}}$, and hence $(w_{i})_{i \in \{0, \dots, m\}}$ can be $\epsilon$-shadowed by the orbit of $T^{m+k+t+2}_{\beta,\alpha}(c^*) \in J$ under $T$, yielding a contradiction.
    \end{proof}

\section{Concluding remarks and open questions}\label{sec:conclusion}

Anther famous class of interval map which fall within our setting are transitive negative \mbox{$\beta$-transformation}, namely maps of the form $x \mapsto -\beta x + \alpha$.  These maps and their dynamical properties have been well studied, see for instance \cite{neg_2,neg_3,neg_1,neg_4} and references therein.  However, to the authors knowledge, a classification of the set of $(\beta,\alpha)$ for which the negative \mbox{$\beta$-transformation} $x \mapsto -\beta x + \alpha$ is transitive still remain open, namely if an analogue of the results of \cite{G1990,PalmerThesis} hold for negative \mbox{$\beta$-transformations}.  In fact, by appropriately modifying our proof, \Cref{thm:beta_shadowing} can be extended to include any piecewise continuous monotone interval map for which one can find an interval $J$ such that the first return map to $J$ is a transitive piecewise monotone interval map. Thus, if an analogue of the results of \cite{G1990,PalmerThesis} holds for negative \mbox{$\beta$-transformations}, then \Cref{thm:beta_shadowing} could be extended to include negative \mbox{$\beta$-transformations}.

Another open question, would be if our proof of \Cref{thm:Shadowing_Main} could be extended to include transitive maps with at least one discontinuity, but where on the partition elements we remove the condition of monotonicity.  For our class of examples, it would also be interesting to investigate if given a transitive piecewise continuous monotone interval map and $\epsilon, \delta \in \mathbb{R}_{>0}$, one can classify all the $\delta$-pseudo orbits which cannot be $\epsilon$-shadowed by a true orbit.

\bibliographystyle{alpha}
\bibliography{ref}

\end{document}